\documentclass[11pt]{amsart}
\usepackage[english]{babel}
\usepackage{amssymb,amsmath,amsthm}
\usepackage[latin2]{inputenc}
\usepackage{amsfonts, verbatim, amscd}
\usepackage[foot]{amsaddr}

\reversemarginpar
\newtheorem{theorem}{Theorem}[section]

\newtheorem{lemma}[theorem]{Lemma}

\theoremstyle{definition}

\def\acknowledgment{\par\addvspace{17pt}\small\rmfamily
\trivlist\if!\ackname!\item[]\else
\item[\hskip\labelsep
{\bfseries\ackname}]\fi}

\def\C{\mathbb{C}}

\numberwithin{equation}{section}

\begin{document}
\title{Approximation theorems for Pascali systems}

\author{Barbara Drinovec Drnov\v sek, Uro\v s Kuzman}

\address[Barbara Drinovec Drnov\v sek]{Faculty of Mathematics and Physics, University of Ljubljana, 
and Institute of Mathematics, Physics and Mechanics, Jadranska 19, 
1000 Ljubljana, Slovenia, barbara.drinovec@fmf.uni-lj.si}

\address[Uro\v s Kuzman]{Faculty of Mathematics and Physics, University of Ljubljana, 
and Institute of Mathematics, Physics and Mechanics, Jadranska 19, 
1000 Ljubljana, Slovenia, uros.kuzman@fmf.uni-lj.si}

%

\begin{abstract}
Based on Runge theorem for generalized analytic vectors proved by Goldschmidt in 1979 we provide a Mergelyan-type and a Carleman-type approximation theorems for solutions of Pascali systems.     
\end{abstract}
\maketitle

\section{Introduction}

Given a domain $\Omega\subset \mathbb{C}$ the \emph{Pascali system} on $\Omega$ is an elliptic system which can be written in the following normalized form:
	\begin{equation}\label{Pascali}
		w_{\overline{\zeta}}+B_1 w + B_2\overline{w}=0,
	\end{equation} 
where $w\colon\Omega\to \mathbb{C}^n$ is a complex vector function, while $B_1$ and $B_2$ are $n\times n$ matrix funcions. In the scalar case $n=1$, the fundamental theory for such systems was developed by Bers \cite{BERS} and Vekua \cite{VEKUA}, thus equation \eqref{Pascali} is often called the \emph{Bers-Vekua equation}. 
In the case $n\geq 2$, the pioneering work was done by Pascali \cite{PASCALI}. Furthermore, solutions of such a system are a subclass of generalized analytic vectors, which correspond to elliptic systems with vanishing Beltrami coefficient \cite{BOJARSKI,GB, Goldschmidt, WENDLAND}. We denote them by $\mathcal{O}_{B}(\Omega)$. In general, their regularity depends on the regularity of $B_1$ and $B_2$ (see the preliminary section), however, in our theorems, we work with smooth matrix functions and thus elements of the set $\mathcal{O}_{B}(\Omega)$ are smooth vector functions. 

We shall focus on approximation theorems which originate from the classical complex function theory (see e.g. \cite{FFW}). 
We rely on a version of Runge approximation theorem for generalized analytic vectors, which was proved 
by Goldschmidt \cite[Theorem 3.1]{Goldschmidt}, and the recent work on Pascali systems by Sukhov and Tumanov \cite[\S 3]{ST}. All our results are valid for $n\geq 1$. 

We prove a Mergelyan-type approximation theorem for maps from compact admissible sets:
Let $\Omega\subset \mathbb{C}$ be an open subset and let $S\subset \Omega$ be a compact subset.
Recall that the set $S$ is \emph{Runge} in $\Omega$ if  $\Omega\setminus S$ has no relatively compact components.
We call $S$ \emph{admissible} if it is Runge in $\Omega$ and if it can be
written as $S=K\cup E$, where $K$ is a finite union of pairwise disjoint 
compact domains in $\mathbb{C}$ with smooth boundary and $E=\overline{S\setminus K}$ is a union of finitely many pairwise disjoint smooth Jordan
arcs and closed Jordan curves meeting $K$ only in their endpoints (if at all) and such that their intersection with the boundary $bK$ is transverse. Our first result is the following Mergelyan-type theorem. 

\begin{theorem} 
\label{Mergelyan}
Let $\Omega\subset \mathbb{C}$ be an open subset, let $S\subset \Omega$ be an admissible compact set, and let $B_1$ and $B_2$ be $n\times n$ matrix functions with coefficents in $\mathcal C^\infty(\Omega)$. Given $\epsilon>0$ and a smooth map  $f\colon S\to \mathbb{C}^n$  whose restriction to the interior $\mathring{S}$ lies in $\mathcal{O}_{B}(\mathring{S})$, 
there exists $w\in\mathcal{O}_{B}(\Omega)$ such that 
$$\left |f(\zeta)-w(\zeta)\right |<\epsilon \text{ for all } \zeta\in S.$$      
\end{theorem} 

\noindent A similar version of such an approximation theorem on admissible sets was recently obtained for conformal minimal immersions, null curves \cite{AF, AFL} and for holomorphic Legendrian curves \cite{F,AFL}. 

Carleman's approximation theorem asserts that given a continuous function $f\colon\mathbb{R}\to \mathbb{C}$ 
and a strictly positive function $\varepsilon\colon \mathbb{R}\to (0,\infty)$ there exists an entire function 
$g\colon\mathbb{C}\to \mathbb{C}$  
such that $|f(\zeta)-g(\zeta)|<\varepsilon(\zeta)$ for every $\zeta\in \mathbb{R}$. An analogue of this statement was recently proved for conformal minimal immersions \cite{CIC}. We provide a version for solutions of Pascali systems.

\begin{theorem}
\label{Carleman}
Let  $B_1$ and $B_2$ be $n\times n$ matrix functions with coefficients in $\mathcal C^\infty(\C)$.
Given a continuous map $f\colon\mathbb{R}\to \mathbb{C}^n$ 
and a strictly positive function $\varepsilon\colon \mathbb{R}\to (0,\infty)$ there exists $w\in\mathcal{O}_{B}(\mathbb{C})$ 
such that 
$$\left| f(\zeta)-w(\zeta)\right| <\epsilon(\zeta) \text{ for all } \zeta\in\mathbb{R}.$$  
\end{theorem}
 
\noindent{\bf{Acknowledgements.}} The authors are supported in part by the research program P1-0291 and grants N1-0137, J1-9104, J1-1690, BI-US/19-21-108 from ARRS, Republic of Slovenia (Javna Agencija za Raziskovalno Dejavnost Republike Slovenije).
 
\section{Preliminaries}
Let $p>2$ and assume that the coefficients of $B_1$ and $B_2$ belong to the Lebesque class $L^p_{loc}(\Omega)$. Then one can consider weak solutions of (\ref{Pascali}) defined as vector functions $w\in L^q_{loc}(\Omega,n)$, $\frac{1}{p}+\frac{1}{q}=1$, satisfying the condition
$$\iint_{\Omega}\;\left(w^T\cdot \varphi_{\bar{\zeta}}- (B_1w+B_2\overline{w})^T\cdot \varphi\right)\;\; d\zeta d\overline{\zeta} =0$$
for all smooth vector functions $\varphi\in \mathcal{C}^\infty_0(\Omega,n)$ with compact support (note that here the minus sign arises from the integration by parts). We denote by  $\mathcal{O}^*_{B}(\Omega)\subset L^q_{loc}(\Omega,n)$ the subset of all such solutions. However, using the standard boothstrapping arguments, it turns out that $\mathcal{O}^*_{B}(\Omega)$ indeed lies in the Sobolev class $W^{1,p}_{loc}(\Omega,n)$. That is, the $L^{p}_{loc}$-regularity of $B_1$ and $B_2$ implies the $W^{1,p}_{loc}$-regularity of $w\in \mathcal{O}^*_{B}(\Omega).$ Similarly, we can take $k\in\mathbb{N}_0$ and $0<\alpha<1$ and introduce the H$\mathrm{\ddot{o}}$lder spaces $\mathcal{C}^{k,\alpha}(\Omega)$ and $\mathcal{C}^{k,\alpha}(\Omega,n).$ Again, the ellipticity of the Pascali system implies that if the coefficients of $B_1$ and $B_2$ belong to $\mathcal{C}^{k,\alpha}(\Omega)$, the solutions of (\ref{Pascali}) which are at least $L^q_{loc}$-regular automatically belong to the class $\mathcal{C}^{k+1,\alpha}(\Omega,n)$. As indicated in the introduction, we denote the set of all differentiable solutions of (\ref{Pascali}) by $\mathcal{O}_{B}(\Omega)$ and remark that in our main theorems the coefficients of $B_1$ and $B_2$ are assumed to be $\mathcal{C}^{\infty}$-regular, hence $\mathcal{O}_{B}(\Omega)\subset\mathcal{C}^{\infty}(\Omega,n)$. For more details on the regularity and boothstrapping methods see e.g. \cite{VEKUA}.

From now on we denote by
$$\bar{\partial}_B(w)=w_{\bar{\zeta}}+B_1w+B_2\bar{w},$$
where $w_{\bar{\zeta}}$ means either the usual derivative with respect to the conjugate complex variable or its weak analogue (it will be clear from the context). Let $\mathcal{D}\subset\mathbb{C}$ be an open disc and let the coefficients of $B_1$ and $B_2$ belong to the class $\mathcal{C}^{k,\alpha}(\mathcal{D})$ for some $k\in\mathbb{N}_0$ and $0<\alpha<1$. Then $\bar{\partial}_B$ can be treated as a linear operator mapping from $\mathcal{C}^{k+1,\alpha}(\mathcal{D},n)$ into the space $\mathcal{C}^{k,\alpha}(\mathcal{D},n)$. For $n=1$ this operator is known to be surjective. Indeed, let $T_\mathcal{D}$ be the classical Cauchy-Green operator for $\mathcal{D}$. Then we can introduce a bijective integral operator $\mathcal{P}_{\mathcal{D}}\colon \mathcal{C}^{k+1,\alpha}(\mathcal{D},n)\to \mathcal{C}^{k+1,\alpha}(\mathcal{D},n)$ given by 
$$\mathcal{P}_{\mathcal{D}}(w)=w+T_{\mathcal{D}}\left(B_1 w+B_2\bar{w}\right).$$
It turns out that $\mathcal{P}_{\mathcal{D}}^{-1}\circ T_{\mathcal{D}}$ is a bounded right inverse for $\bar{\partial}_B$. For $n\geq 2$, an analogue of this fact was proved only recently. In particular, in the higher dimensional case $\mathcal{P}_{D}$ may have a non-trivial kernel and has to be slightly perturbed before it can be inverted \cite[Theorem 3.1, Corolary 3.6.]{ST}.

If we combine the surjectivity of $\bar{\partial}_B$ and the extension theory for H$\mathrm{\ddot{o}}$lder spaces we obtain the following theorem concerning the approximation of vector functions with small $\bar{\partial}_B$-derivatives.

\begin{theorem}\label{t2.2}
Let $\mathcal{D}\subset\mathbb{C}$ be an open disc and assume that the coefficients of $B_1$ and $B_2$ belong to the class $\mathcal{C}^{k,\alpha}(\mathcal{D})$. Given $\epsilon>0$ there exists a constant $\delta>0$ such that given any smoothly bounded domain $\Omega\Subset\mathcal{D}$ and $g\in\mathcal{C}^{k+1,\alpha}(\overline{\Omega})$ with $\left\|\bar{\partial}_B(g)\right\|_{\mathcal{C}^{k,\alpha}(\Omega)}<\delta$ there exists $w\in \mathcal{O}_{B}(\Omega)$ such that $$\left\|w-g\right\|_{\mathcal{C}^{k,\alpha}(\Omega)}<\epsilon.$$  
\end{theorem}
\begin{proof} 
Note that $g$ is assumed to be regular up to the boundary of $\Omega$. Hence, we can introduce extension operators $E_{k,\alpha}\colon \mathcal{C}^{k,\alpha}(\overline{\Omega},n)\to \mathcal{C}^{k,\alpha}(\mathbb{C},n).$ It is crucial that, by \cite[Theorem 4, p.177]{Stein}, they can be chosen so that their norms are bounded by constants that are independent of $\overline{\Omega}$. In particular, there exist a constant $C_{k,\alpha}\geq 1$ such that
$$\left\|E_{k,\alpha}\left(\bar{\partial}_B(g)\right)\right\|_{\mathcal{C}^{k,\alpha}(\mathcal{D})}\leq C_{k,\alpha}\; \cdot\left\|\bar{\partial}_B(g)\right\|_{\mathcal{C}^{k,\alpha}(\overline{\Omega})}$$
for any corresponding domain $\Omega\Subset\mathcal{D}$.

Since the operator $\bar{\partial}_B\colon \mathcal{C}^{k+1,\alpha}(\mathcal{D},n)\to \mathcal{C}^{k,\alpha}(\mathcal{D},n)$ is surjective, there exist a solution $u\in\mathcal{C}^{k+1,\alpha}(\mathcal{D})$ of the non-homogeneous equation $$\bar{\partial}_B(u)=E_{k,\alpha}\left(\bar{\partial}_B(g)\right).$$ 
We restrict it to $\Omega$ and define $w=g-u|_{\Omega}.$ This is the map we seek since $w\in \mathcal{O}_{B}(\Omega)$ and 
$$\left\|w-g\right\|_{\mathcal{C}^{k,\alpha}(\Omega)}<C\cdot \delta,$$
where the constant $C>0$ arises as a product of $C_{k,\alpha}$ and the operator norm of the bounded right inverse of $\bar{\partial}_B$ on $\mathcal{D}$.
\end{proof}




We conclude this preliminary section by presenting Runge-type theorem for Pascali systems provided by Goldschmidt \cite{Goldschmidt}. For the convenience of the reader, we include its proof. The proof relies on the Riesz representation theory for Lebesque spaces, hence it is essential to work with weak solutions of (\ref{Pascali}). However, as explained above, higher regularity of $B_1$ and $B_2$ automatically improves the regularity of solutions.

\begin{theorem}
\label{osem} 
Let $U\subset \C$ be a domain and assume that the coefficients of $B_1$ and $B_2$ belong to the class $L^p_{loc}(U,n)$, $p>2$. 
Let $K\subset U$ be a compact Runge subset. Given $\epsilon>0$ and $f\in\mathcal{O}^*_{B}(\Omega)$ for some neighborhood $\Omega$ of $K$, there exists  $w\in\mathcal{O}^*_{B}(U)$ such that $$\left\|f-w\right\|_{L^\infty(K)}<\epsilon.$$
\end{theorem}

\begin{proof} 
Fix  $\epsilon>0$ and choose $f\in\mathcal{O}^*_{B}(\Omega)$ as in the theorem. If necessary we may shrink the neighborhood $\Omega$ so that we still have $K\subset \Omega$ and that $\Omega$ admits no relatively compact connected components in $U$.

Given $w\in L_{loc}^q(U,n)$, $\frac{1}{p}+\frac{1}{q}=1$, we define
$$\overline{\partial}_B^*(w)=w_{\bar{\zeta}}-B_1^Tw-\overline{B}_2^T\bar{w}.$$ 
Furthermore, let $\varphi\in \mathcal{C}^{\infty}_0(U,n)$. We have
$$\varphi^T\cdot \overline{\partial}_{B}(w)+w^T\cdot \overline{\partial}_{B}^*(\varphi)=(\varphi^T\cdot w_{\bar{\zeta}}+w^T\cdot \varphi_{\bar{\zeta}})+(\varphi^T\cdot B_2 \bar{w}-\overline{\varphi}^T\cdot \overline{B}_2w).$$ 
Hence, applying the integration by parts, we get
$$\iint_{U}\left(\varphi^T\cdot \overline{\partial}_{B}(w)+w^T\cdot \overline{\partial}_{B}^*(\varphi)\right)\;d\zeta d\overline{\zeta}=
\iint_{U}\left(\varphi^T\cdot B_2 \bar{w}-\overline{\varphi}^T\cdot\overline{B}_2w\right)\;d\zeta d\overline{\zeta}.$$
This last integral is imaginary, thus we have
\begin{equation*}\label{re2}
\textrm{Re}\left(\iint_{U}\varphi^T\cdot \overline{\partial}_{B}(w)+w^T\cdot \overline{\partial}_{B}^*(\varphi)\;d\zeta d\overline{\zeta} \right)=0,\; \forall \varphi\in \mathcal{C}^{\infty}_0(U,n).
\end{equation*}
This implies that for
 $w\in L_{loc}^q(U,n)$ we have
\begin{equation}
\label{re}
w\in \mathcal{O}^*_{B}(U)\iff \textrm{Re}\left(\iint_{U}\; w^T\cdot \overline{\partial}_B^*(\varphi)\;d\zeta d\overline{\zeta}\right)=0,\; \forall \varphi\in \mathcal{C}^{\infty}_0(U,n),
\end{equation}
where the reverse implication follows from the fact that if $\varphi\in \mathcal{C}^{\infty}_0(U,n)$,
then we also have $\imath\varphi\in \mathcal{C}^{\infty}_0(U,n)$.

Given $w\in \mathcal{O}^*_{B}(U)$ its restriction $w|_{\Omega}$ belongs to $\mathcal{O}^*_{B}(\Omega)$. We denote by
$\mathcal{O}^*_{B}(U)|_{\Omega}$ the set of all such restrictions. We would like to prove that the closure of this set in the topology of the space $L^{q}_{loc}(\Omega,n)$ is equal to the space  $\mathcal{O}^*_{B}(\Omega)$. If this holds, the desired conclusion follows from the boothstrapping argument. That is, there exists an extension $w$ of $f$ which is $L^q_{loc}$-close to $f$ on $\Omega$ and therefore $W^{1,p}$-close to $f$ on some smaller neighborhood $\Omega'$ of $K$ with smooth boundary. By the Sobolev embedding theorem such a function is $L^{\infty}$-close to $f$ on $\Omega'$ and therefore $L^\infty$-close to $f$ on $K$. 

Assume the contrary. Then, by Hahn-Banach theorem, there is a bounded linear functional $\mu\colon L^{q}_{loc}(\Omega,n)\to \mathbb{C}$ that vanishes on $\mathcal{O}^*_{B}(U)|_{\Omega}$ and an element $w_1\in  \mathcal{O}^*_{B}(\Omega)\setminus\mathcal{O}^*_{B}(U)|_{\Omega}$ such that $\mu(w_1)\neq 0$. In particular, we may assume that $\textrm{Re}\left(\mu (w_1)\right)\neq 0$ (otherwise we take $\imath\mu$). Furthermore, we can represent $\mu$ by some $g\in L^p(\Omega,n)$ such that $\textrm{supp}\;g\Subset\Omega$. Let us extend $g$ to $U\setminus \Omega$ by setting $g=0$. We have
$$\mu(w)=\iint_{\Omega} w^T\cdot g \;\;d\zeta d\overline{\zeta} =\iint_{U} w^T\cdot g \;\;d\zeta d\overline{\zeta},\;\; w\in L^{q}_{loc}(\Omega,n).$$

 Recall that $w\in \mathcal{O}^*_{B}(U)$ precisely when
$$\textrm{Re}\left(\iint_{U}\; w^T\cdot \overline{\partial}_{B}^*(\varphi)\;\;d\zeta d\overline{\zeta}\right)=0,\; \forall \varphi\in \mathcal{C}^{\infty}_0(U,n).$$
Hence the map $g\in L^p(U,n)$ belongs to the closure of the set $\overline{\partial}_{B}^*\left(\mathcal{C}^{\infty}_0(U,n)\right)$ in the topology of the space $L^{p}(U,n)$. Therefore, there exists a sequence $\left\{\varphi_m\right\}\subset \mathcal{C}^{\infty}_0(U,n)$ with $\textrm{supp}\;\overline{\partial}_{B}^*\left(\varphi_m\right)\Subset\Omega$ and such that $$\lim_{m\to\infty}{\bar{\partial}}_{B}^*\left(\varphi_m\right)= g.$$ 
We claim that $\varphi_m\in \mathcal{C}^{\infty}_0(\Omega,n)$. That is, if the supports of $\varphi_m$ are compact in $U$, they must be compact in $\Omega$ as well. Indeed, note that $\overline{\partial}_{B}^*\left(\varphi_m\right)=0$ on $U\setminus\Omega$. Moreover, $\textrm{supp}\;\varphi_m\Subset U$ implies that $\varphi_m$ vanishes on some open set. However, by the Similarity principle for Pascali systems \cite{BUCHANAN}, every solution of \eqref{Pascali} can be represented as a product of an invertible matrix function and a holomorphic vector function. This together with the fact that $U\setminus\Omega$ admits no relatively compact components implies that $\varphi_m=0$ in $U\setminus\Omega.$ 

For $w_1\in \mathcal{O}^*_{B}(\Omega)\setminus\mathcal{O}^*_{B}(U)|_{\Omega}$ chosen above we have 
$$\lim_{m\to\infty}\textrm{Re}\left(\iint_{\Omega} w_1^T\cdot \overline{\partial}_{B}^*(\varphi_m) \;\;d\zeta d\overline{\zeta}\right)= \textrm{Re}\left(\iint_{\Omega} w_1^T\cdot g \;\;d\zeta d\overline{\zeta}\right)\neq 0.$$
But then there is $l\in\mathbb{N}$ such that
$$\textrm{Re}\left(\iint_{\Omega} w_1^T\cdot \overline{\partial}_{B}^*(\varphi_l) \;\;d\zeta d\overline{\zeta}\right)\neq 0.$$
By an analogue of (\ref{re}) this contradicts the fact that $w_1\in \mathcal{O}^*_{B}(\Omega)$.
\end{proof}

\section{Proofs of Theorem \ref{Mergelyan} and Theorem \ref{Carleman}}
We begin by proving a local Mergelyan-type lemma for bounded domains.
\begin{lemma}
\label{Merg}
Let $K\subset \mathbb{C}$ be a smoothly bounded compact domain. Given $\epsilon>0$ and a smooth map  $f\colon K\to \mathbb{C}^n$  whose restriction to the interior $\mathring K$ is in $\mathcal{O}_{B}(\mathring K)$ there exist
an open neighbourhood $U$ of $K$ and $w\in\mathcal{O}_{B}(U)$ such that 
$$\left\|f-w\right\|_{L^\infty(K)}<\epsilon.$$     
\end{lemma}
\begin{proof}
Let $\mathcal{D}\subset\mathbb{C}$ be an open disc containing $K$. We apply the operator $E_{1,\alpha}$ from the proof of Theorem \ref{t2.2} and extend $f$ into $\tilde{f}=E_{1,\alpha}(f)\in\mathcal{C}^{1,\alpha}(\mathcal{D},n)$. Let $\left\{\Omega_m\right\}$ be a sequence of smoothly bounded domains that are compactly contained in $\mathcal{D}$ and shrink towards $K$. We denote by $f_m$ the restrictions of $\tilde{f}$ to $\Omega_m$. Since $f$ is smooth on $K$ and $f|\mathring {K}\in\mathcal{O}_{B}(\mathring{K})$ we have
$$\lim_{m\to\infty}\left\|\overline{\partial}_{B}(f_m)\right\|_{\mathcal{C}^{0,\alpha}(\Omega_m)}=0.$$
Together with Theorem \ref{t2.2} this yields the desired conclusion. That is, we can set $U=\Omega_m$ for some large $m\in\mathbb{N}.$
\end{proof}

\begin{lemma}
\label{arcs}
Let $\gamma\colon [a,b] \to \mathbb{C}$ be a parametrization of a smooth Jordan arc or a closed Jordan curve.
Given a smooth map $f\colon \gamma([a,b])\to \mathbb{C}^n$
there exist an open neighbourhood $U$  of $\gamma([a,b])$ 
and a smooth map $F$ on $U$ which solves the Pascali system (\ref{Pascali}) on 
$\gamma([a,b])$ and satisfies $F|_{\gamma([a,b])}\equiv f$.
\end{lemma}

\begin{proof}
We only consider Jordan arcs, the proof for closed Jordan curves proceeds similarly.
We can extend $\gamma$ smoothly to $[a-\eta,b+\eta]$ for some $\eta>0$. 
For any $x\in (a-\eta,b+\eta)$ there are a neighbourhood $U_x$ of $\gamma(x)$ in 
$\mathbb{C}$ and a smooth map $f_x\colon U_x \to \mathbb{C}^n$, which agrees with $f$ on 
$U_x\cap \gamma([a,b])$ and solves (\ref{Pascali}) on $U_x\cap \gamma((a-\eta,b+\eta))$. Indeed, locally near $\gamma(x)$, the set $\gamma  (a-\eta,b+\eta)$ is a graph above one of the coordinate axes. Without loss of generality, let us assume that this is the $x$-axis. Then there are a neighbourhood  $U_x=I_x\times J_x$
of $\gamma(x)$ in $\mathbb{C}$ and a map $\psi\colon I_x\to J_x$ such that
$$U_x\cap \gamma((a-\eta,b+\eta))=\{(t+\imath \psi(t)\colon t\in I_x\}.$$ 
An extension of $f$ to $U_x$ which solves (\ref{Pascali}) on $U_x\cap \gamma((a-\eta,b+\eta))$ can be found in the form
$f_x(t+\imath s)=f(t+\imath \psi(t))+f_1(t)(s-\psi(t))$, where $f_1$ is an appropriate function defined on $I_x$.

Let $U=\cup_{x\in\gamma([a,b])}U_x$ and let $\{\varphi_x\}$ be a smooth partition of unity subordinate to the covering 
$\{U_x\}_{x\in\gamma([a,b])}$. We define the map $F$ on $U$ by 
$$F=\sum_x \varphi_x f_x.$$
Then we have $F|_{\gamma([a,b])}\equiv f$ and since $f_x$ solves (\ref{Pascali}) on $U_x\cap \gamma([a,b])$, we obtain the following equality on $\gamma([a,b])$:
$$\bar{\partial}_B(F)=\sum_x (\varphi_x)_{\bar{\zeta}}f_x+\sum_x \varphi_x\bar{\partial}_B(f_x)=\sum_x (\varphi_x)_{\bar{\zeta}}f_x = f \left(\sum_x \varphi_x\right)_{\bar{\zeta}}=0.$$
That is, $F$ solves the Pascali system (\ref{Pascali}) on $\gamma([a,b])$.
\end{proof}

\begin{proof}[Proof of Theorem \ref{Mergelyan}]
Fix $\epsilon>0$ and recall that $S=K\cup E$,
where $K$ is a finite union of pairwise disjoint 
compact domains in $\mathbb{C}$ with smooth boundary and $E=\overline{S\setminus K}$ is a union of finitely many pairwise disjoint smooth Jordan
arcs and closed Jordan curves. 
By Lemma \ref{Merg} there exist 
an open neighbourhood $U$ of $K$ and $g\in\mathcal{O}_{B}(U)$ such that 
$$\left\|f-g\right\|_{L^\infty(K)}<\frac \epsilon 4.$$     
Since $f$ and $g$ are close on $K$ and continuous on $U\cap E$, 
and since Jordan curves from $E$ meet the boundary $bK$ transversally,
there is an open neighbourhood $U_1$ of $K$, 
$K\subset \overline U_1\subset U$, such that 
$$|f(\zeta)-g(\zeta)|<\frac \epsilon 2  \text{ for all } \zeta\in E\cap U_1.$$
We now glue $f$ and $g$ on $E\cap U_1$ by a smooth cut off function. That is, we choose a smoothly bounded neighbourhood $U_2$ of $K$, $K\subset \overline U_2\subset U_1$, and a smooth cut off function $\varphi\colon \mathbb{C}\to [0,1]$ 
that satisfies $\varphi|_{U_2}\equiv 1$ and ${{\rm supp}\,} \varphi\subset U_1$. We
define a smooth map $h\colon U_2\cup E\to \mathbb{C}^n$  by
$$h=\varphi \;g+(1-\varphi)f.$$ 
Note that such a map satisfies $h\equiv g$ on $U_2$, $h\equiv f$ on ${E\cap (\mathbb{C}\setminus \overline U_1)}$, and 
$$|f(\zeta)-h(\zeta)|<\frac {\epsilon} 2 \text{ for all } \zeta\in S.$$

We now use Lemma \ref{arcs} to get an open neighbourhood $V$ of $E$ 
and a smooth map $F$ on $V$ which solves (\ref{Pascali}) on 
$E$ and satisfies $F|_E\equiv h$. We glue $F$ with $h$ in order to obtain a map that admits a vanishing $\bar{\partial}_B$-derivative on $S$. Let $U_3$ be a smoothly bounded neighbourhood of $K$, 
$K\subset \overline U_3\subset U_2$, and let $\varphi\colon \mathbb{C}\to [0,1]$ be a smooth cut off function that satisfies $\varphi|_{U_3}\equiv 1$ and ${{\rm supp}\,} \varphi\subset U_2$.
Let a smooth map $H\colon U_3\cup V\to \mathbb{C}^n$ be given by
$$H=\varphi \;h+(1-\varphi)F.$$ 
We have 
$$|H(\zeta)-f(\zeta)|<\frac {\epsilon} 2 \text{ for all } \zeta\in S.$$
On $E$ we have $h=F$, $\bar{\partial}_B(h)=\bar{\partial}_B(F)=0$ and therefore
$$\bar{\partial}_B(H)=\varphi_{\bar{\zeta}}(h-F)+\varphi \;\bar{\partial}_B(h)+(1-\varphi)\bar{\partial}_B(F)=0.$$
Furthermore, the fact that $H\equiv h\equiv g$ on $U_3$ implies that $H\in {\mathcal O}_{B}(U_3)$. In particular, for every $\zeta\in S$ we have $\bar{\partial}_B(H)(\zeta)=0$.

We now proceed as in the proof of Lemma \ref{Merg} and set $\left\{\Omega_m\right\}$ to be a sequence of smoothly bounded domains that are compactly contained in $U_3\cup V$ and shrink towards $S$. We denote by $H_m$ the restriction of $H$ to the set $\Omega_m$. We have
$$\lim_{m\to\infty}\left\|\bar{\partial}_B(H_m)\right\|_{\mathcal{C}^{0,\alpha}(\Omega_m)}=0.$$
Therefore, by Theorem \ref{t2.2}, for $m\in\mathbb{N}$ that is large enough we can find $w_m\in\mathcal{O}_{B}(\Omega_m)$ such that
$$\left\|w-H_m\right\|_{\mathcal{C}^{1,\alpha}(\Omega_m)}<\frac{\epsilon}{4}.$$
Finally, by Runge theorem \ref{osem}, there exists $w\in \mathcal{O}_{B}(\Omega)$ such that
$$\left\|w-w_m\right\|_{L^\infty(S)}<\frac{\epsilon}{4}.$$
This is the map we seek. 
\end{proof}

Finally, we prove Carleman theorem using Mergelyan's theorem for 
Pascali systems following the usual inductive construction
\cite{FFW}.

\begin{proof}[Proof of Theorem \ref{Carleman}]
Without loss of generality we may assume that the starting map $f$ is smooth (otherwise work with its smooth approximation).
For $m\in\mathbb{N}$ we define the sets
$$S_m=\{\zeta\in\mathbb{C}\colon|\zeta|\le m\}\cup [-m-2,m+2],$$ 
$$\textstyle{\Omega_m=\{\zeta\in\mathbb{C}\colon|\zeta|< m+\frac 13\},}$$
and a decreasing sequence
$$\varepsilon_m=\min\{\varepsilon(\zeta)\colon |\zeta|\le m+2\}.$$

Let $f_0(\zeta)=f(\zeta)$ for  $\zeta\in \mathbb{R}$.
We construct inductively
a sequence of smooth maps $f_m\colon \Omega_m\cup \mathbb{R}\to \mathbb{C}^n$, $m\geq 1$, which satisfies the following properties:
\begin{enumerate}
\item[i)] $f_m\in \mathcal{O}_{B}(\Omega_m)$,
\item[ii)] $f_m(\zeta)=f(\zeta)$ for $\zeta\in \mathbb{R}$ such that $|\zeta|\ge m+\frac 23$,
\item[iii)] $|f_m(\zeta)-f_{m-1}(\zeta)|<\frac{\varepsilon_{m-1}}{2^{m+1}}$ for all
$\zeta\in S_{m-1}$.
\end{enumerate}
Let us fix $m\ge 1$ and assume that we have already constructed the map $f_{m-1}$. The set $S_{m-1}$ is admissible. Hence, by Theorem  \ref{Mergelyan}, there is $g_m\in \mathcal{O}_{B}(\mathbb{C})$ such that for every $\zeta\in S_{m-1}$ we have the following estimate: $$|g_m(\zeta)-f_{m-1}(\zeta)|<\frac{\varepsilon_{m-1}} {2^{m+1}}.$$
Let $\varphi_m\colon \mathbb{C}\to [0,1]$ be a smooth cut off function
such that $\varphi_m\equiv 1$ on $\Omega_m$ and $\varphi_m(\zeta)=0$ for $|\zeta|\ge m+\frac 23$. We define $f_m\colon \Omega_m\cup\mathbb{R}\to\mathbb{C}^n$ by setting
$$f_m=\varphi_m \;g_m+(1-\varphi_m)f_{m-1}.$$ 
Note that such a function satisfies the properties i), ii) and iii).

It follows from iii) that for each $m,k\ge 1$ and $\zeta \in S_m$ we have
\begin{equation}
|f_{m+k}(\zeta)-f_m(\zeta)|\le\sum_{i=1}^k |f_{m+i}(\zeta)-f_{m+i-1}(\zeta)|\le
\sum_{i=1}^k\frac{\varepsilon_{m+i-1}}{2^{m+i+1}}\le \frac{\varepsilon_{m}}{2^{m+1}}  . \label{pr1}
\end{equation}
Therefore the sequence $f_m$ converges uniformly on compact sets in $\mathbb{C}$. This, together with the fact that $f_m\in\mathcal{O}_B(\Omega_m)$, implies the existence of its limit $w\in  \mathcal{O}_{B}(\mathbb{C})$ (the uniform convergence implies that $w$ is a weak solution of (\ref{Pascali}) on every compact subset of $\mathbb{C}$). Finally, by ii) and  (\ref{pr1}) we have for every $\zeta$,  $|\zeta|\in[m+1,m+2]$, the following estimate:
$$|w(\zeta)-f(\zeta)|=|w(\zeta)-f_m(\zeta)|\le \sum_{i=1}^\infty\frac{\varepsilon_{m+i-1}}{2^{m+i+1}}\le \frac{\varepsilon_{m}}{2^{m+1}}<\varepsilon(\zeta).$$
Hence $w$ is the vector function we seek.
\end{proof}

\end{document}